\documentclass[11pt,a4paper]{article}
\usepackage[utf8]{inputenc}
\usepackage[T1]{fontenc}
\usepackage[english]{babel}
\usepackage{lmodern}
\usepackage{authblk}
\usepackage{amsmath,amssymb,amsthm}
\usepackage{mathtools}
\usepackage{mathrsfs}
\usepackage[margin=2.5cm]{geometry}
\usepackage{enumitem}
\usepackage{microtype}
\usepackage[hidelinks]{hyperref}
\hypersetup{
  pdftitle={Markovization of Randomized Stopping Times},
  pdfauthor={K. O. Sokolov},
  pdfkeywords={randomized stopping time, Markovian intensity, relative entropy, occupation measure, Markovian projection, Skorokhod embedding}
}
\usepackage[nameinlink,noabbrev]{cleveref}
\newtheorem{theorem}{Theorem}[section]
\newtheorem{proposition}[theorem]{Proposition}
\newtheorem{lemma}[theorem]{Lemma}
\newtheorem{corollary}[theorem]{Corollary}
\theoremstyle{definition}
\newtheorem{assumption}[theorem]{Assumption}
\newtheorem{definition}[theorem]{Definition}
\newtheorem{example}[theorem]{Example}
\newtheorem{remark}[theorem]{Remark}
\newcommand{\R}{\mathbb R}
\newcommand{\E}{\mathbb E}
\newcommand{\1}{\mathbf 1}
\newcommand{\F}{\mathcal F}
\newcommand{\Borel}{\mathcal B}

\newcommand{\RST}{\mathrm{RST}}
\newcommand{\KL}{\mathrm{KL}}
\newcommand{\Law}{\mathrm{Law}}
\newcommand{\ri}{\operatorname{ri}}
\newcommand{\cx}{\preceq_{\mathrm{cx}}}

\newcommand{\aff}{\operatorname{aff}}
\newcommand{\conv}{\operatorname{conv}}
\title{Markovization of Randomized Stopping Times}
\author[1]{K. O. Sokolov}
\affil[1]{Lomonosov Moscow State University}
\date{}
\begin{document}
\maketitle

\begin{abstract}
We study randomized stopping times for a Markov process, described by a
progressively measurable intensity $\alpha_t(\omega)$. We prove that every
admissible intensity has a Markovian representative $\lambda(t,X_t)$,
explicitly obtained from the observed measure and the surviving
occupation measure. This representative preserves both the family of
surviving mass measures and the joint distribution of the stopping time and
the state at stopping. Whenever the relative entropy with respect to a
reference intensity $r(t,x)$ is finite, we prove an exact decomposition
showing that Markovization does not increase the entropy.

As a consequence, variational problems in which a randomized stopping
time enters only through its observed measure and an entropy penalty can be
reduced to optimization over Markovian intensities $\lambda(t,x)$. For every
observed measure admitting a finite-entropy representative, there is a unique
minimum-entropy randomized stopping time, and it is Markovian. We discuss the connection with optimal
Skorokhod embedding and prove a finite-constraint realization result for
Brownian stopping. We also approximate arbitrary observed measures by
those generated by bounded Markovian intensities.
\end{abstract}


\section{Introduction}

Optimal stopping is a central topic in stochastic analysis and stochastic
control. A decision time is chosen using the observed history of a
process in order to optimize a given reward \cite{PeskirShiryaev2006}.
The admissible class can also be enlarged to randomized stopping times. These assign to each path a probability distribution on
$[0,\infty)$ whose distribution function is adapted to the observed
filtration. Equivalent forms of this randomization are used in game theory
\cite{TouziVieille2002,LarakiSolan2005,NeumannRamseySzajowski2002}.

Randomized stopping times are also used in the Skorokhod embedding problem. Since the classical constructions
\cite{Skorokhod1965,Root1969,AzemaYor1979}, this problem has provided many
examples of nontrivial stopping rules. A survey of the different
solutions and their properties is given in \cite{Obloj2004}. In the
optimal transport approach, randomized stopping times are used as an
admissible class. The geometry of optimal embeddings is described by
analogues of monotonicity in the Monge--Kantorovich problem
\cite{BeiglbockCoxHuesmann2017}. Related work studies stopping problems
with constraints on the distribution of the stopping time or the
terminal state
\cite{BayraktarMiller2019,BeiglbockEderElgertSchmock2018}.

Randomized stopping can also be formulated through a controlled
stopping intensity. The intensity determines the conditional law of the
stopping time and turns the stopping problem into a control problem. This
approach is used for controlled diffusions in \cite{GyongySiska2008} and
in work on reinforcement learning for optimal stopping \cite{Dong2024}.
Entropy-regularized randomized stopping is also studied through a singular
control formulation in \cite{DianettiFerrariXu2026}. In this paper, the
regularization is defined by relative entropy with respect to a reference
stopping intensity.

Occupation-measure methods provide a second point of contact. For
continuous-time Markov processes, Cho and Stockbridge
\cite{ChoStockbridge2002} formulate optimal stopping as a linear program
over an occupation measure and the joint distribution of the stopping
time and stopping location. More generally, Kurtz and Stockbridge
\cite{KurtzStockbridge1998} develop Markov-control reductions for
controlled martingale problems. These results provide the occupation-measure background for our
construction. For a fixed underlying Markov process, we start from a
possibly path-dependent stopping intensity and identify an explicit
time--state intensity that reproduces the full family of surviving mass
measures and the joint time--state stopping measure.

In this paper, a randomized stopping time is called \emph{Markovian} if
its intensity depends only on the current time and state, and thus has
the form $\lambda(t,X_t)$. Throughout, the term \emph{Markovian} refers
to this intensity-based class. In discrete time,
the analogous rule stops with a probability that depends on the current
state \cite{ChristensenLindensjoNeumann2025}. In continuous time, we work
with intensities with respect to Lebesgue time, as in
\cite{GyongySiska2008,Dong2024}.

Different randomized stopping times can have the same \emph{observed
measure}
$$
  \pi_\xi:=(t,X_t)_\#\xi
$$
on $[0,\infty)\times E$, while depending on the past in different ways.
We seek a representative of a prescribed observed measure whose
intensity is a Borel function of time and state. Given an admissible
path-dependent intensity, we construct such a representative explicitly.
The construction preserves the full family of surviving mass measures and
the observed measure. We then prove an exact relative entropy identity with a
nonnegative remainder. For every fixed observed measure admitting a
finite-entropy representative, this identity identifies a unique
minimum-entropy representative.

Fix a positive Borel intensity $r(t,x)$ satisfying the conditions stated
below, and let $\xi_r$ be the associated reference randomized stopping
time. We measure the difference between $\xi$ and $\xi_r$ by the relative
entropy $\KL(\xi\|\xi_r)$. This is analogous to the Schr\"odinger problem, where one selects a
law closest to a reference Markov law under marginal constraints
\cite{Schrodinger1931,Follmer1988,Leonard2014}. A related idea is used in
entropic optimal transport \cite{Cuturi2013,PeyreCuturi2019}. In our
setting the state dynamics are fixed, and relative entropy regularizes
the randomized stopping rule.

For an admissible intensity, write
$$
  \xi_\alpha(d\omega,dt)
  =S_t^\alpha(\omega)\alpha_t(\omega)\,dt\,P_\mu(d\omega),
  \qquad
  S_t^\alpha=\exp\left(-\int_0^t\alpha_s\,ds\right).
$$
We prove the entropy formula
$$
  \KL(\xi_\alpha\|\xi_r)
  =\E_{P_\mu}\int_0^\infty
       S_t^\alpha\ell_{r(t,X_t)}(\alpha_t)\,dt,
  \qquad
  \ell_a(z)=z\log\frac za-z+a.
$$
The equality is understood in $[0,\infty]$.

Our main result states that every admissible progressively measurable
intensity $\alpha$ has a Borel Markovian representative
$\lambda^\alpha(t,x)$. More precisely, if
$$
  L_t^\alpha(A)=\E_{P_\mu}[S_t^\alpha\1_{\{X_t\in A\}}],
  \qquad
  \overline L^\alpha(dt,dx)=dt\,L_t^\alpha(dx),
$$
then $\pi_\alpha\ll\overline L^\alpha$ and one may take
$$
  \lambda^\alpha
  =\frac{d\pi_\alpha}{d\overline L^\alpha}.
$$
The resulting rule satisfies
$$
  L_t^{\lambda^\alpha}=L_t^\alpha\quad(t\ge0),
  \qquad
  \pi_{\xi_{\lambda^\alpha}}=\pi_{\xi_\alpha}.
$$
If $\KL(\xi_\alpha\|\xi_r)<\infty$, then
$$
  \KL(\xi_\alpha\|\xi_r)
  =\KL(\xi_{\lambda^\alpha}\|\xi_r)
  +\E_{P_\mu}\int_0^\infty
      S_t^\alpha\ell_{\lambda^\alpha(t,X_t)}(\alpha_t)\,dt.
$$
The last term is nonnegative. It vanishes exactly when the original
intensity equals $\lambda^\alpha(t,X_t)$ almost everywhere with respect
to $S_t^\alpha\,dt\,P_\mu(d\omega)$. This gives a quantitative
Pythagorean-type entropy identity, in the spirit of \cite{Csiszar1975}.

The reduction is also related to Markovian mimicking. Classical
mimicking results \cite{Gyongy1986,BrunickShreve2013} replace the state
process while preserving selected marginal distributions. Related
Markovian reductions for killed processes appear in
\cite{CarmonaLauriereLions2024,CarmonaLacker2026}: the former treats
open-loop controls with soft killing, and the latter studies Markovian
projection for It\^o processes conditioned on survival under hard
killing. Our projection acts on the stopping intensity, leaves the underlying
Markov process fixed, and preserves the full family
$(L_t^\alpha)_{t\ge0}$ together with the joint law of the stopping time
and the state at stopping. The exact relative entropy decomposition quantifies this projection
and yields the variational reduction in Section~\ref{sec:variational}.

Consequently, if an objective depends on the stopping rule only through
$\pi_\xi$ and an entropy penalty relative to $\xi_r$, its infimum can be
computed over Markovian intensities. This replaces optimization over
path-dependent rules by optimization over Borel functions on
$[0,\infty)\times E$. The same method applies to path-dependent criteria after adding
suitable coordinates to the state, provided that the augmented process
satisfies Assumption~\ref{ass:markov-class}.

For each observed measure that admits a finite-entropy stopping rule,
relative entropy selects a unique Markovian minimizer. Uniqueness follows
from strict convexity with respect to the stopping measure.

The framework also applies when the terminal law is specified only through
finitely many linear constraints. Section~\ref{sec:brownian}
shows that, for Brownian motion, such constraints can be realized by a
Markovian intensity when the prescribed vector lies in the relative
interior of the convex-order moment region. For example, the functions
$g_i(x)=(x-K_i)^+$ give finite terminal constraints that also arise in
Skorokhod embedding and martingale optimal transport
\cite{Hobson1998,Hobson2011,BeiglbockHenryLaborderePenkner2013,
GalichonHenryLabordereTouzi2014}.

The general results are proved for Markov processes with c\`adl\`ag paths
in a Polish state space. Markovization preserves the observed measure.
Example~\ref{ex:markovization-limit} shows that it may change the law of
the stopped path.

Section~\ref{sec:preliminaries} introduces the measure-valued balance
relations and the entropy formula. Section~\ref{sec:markovization}
proves the Markovization theorem and the exact entropy decomposition.
Section~\ref{sec:variational} gives the variational consequences,
Section~\ref{sec:smoothing} proves approximation by bounded Markovian
intensities, and Section~\ref{sec:brownian} treats finite-dimensional
Brownian terminal constraints.

\section{Basic definitions and preliminary results}
\label{sec:preliminaries}

Let $E$ be a Polish space and let $\Omega=D([0,\infty),E)$ be the space
of right-continuous paths with left limits, equipped with the Skorokhod
topology and its Borel $\sigma$-algebra. Write
$X_t(\omega)=\omega(t)$ for the canonical process and
$$
  \F_t^0=\bigcap_{u>t}\sigma(X_r:0\le r\le u)
$$
for its right-continuous natural filtration.
We use $B_b(E)$ for bounded Borel functions, $B_b^+(E)$ for their
nonnegative members, and $\mathcal P(E)$ for probability measures.

\begin{assumption}[Canonical Markov family]
\label{ass:markov-class}
For each $(s,x)\in[0,\infty)\times E$, a probability measure $P^{s,x}$
on $\Omega$ is given. The family is Borel measurable in $(s,x)$, and
under $P^{s,x}$ the path is equal to $x$ on $[0,s]$ almost surely.
There is a transition function $P_{s,t}(x,dy)$, $0\le s\le t$, such that
$$
  P^{s,x}(X_t\in A)=P_{s,t}(x,A),
  \qquad A\in\Borel(E).
$$
For every $0\le s\le u\le t$ and $f\in B_b(E)$,
$$
  \E^{s,x}[f(X_t)\mid\F_u^0]=P_{u,t}f(X_u)
  \quad P^{s,x}\text{-a.s.},
  \qquad
  P_{s,t}f(x)=\int_E f(y)P_{s,t}(x,dy).
$$
For each $f\in B_b(E)$, the function $(s,t,x)\mapsto P_{s,t}f(x)$
is Borel measurable.

We also assume the following path-segment form of the Markov property.
For every bounded nonnegative Borel function $\varphi$ of the path on
$[u,t]$,
$$
  \E^{s,x}[\varphi((X_r)_{u\le r\le t})\mid\F_u^0]
  =\E^{u,X_u}[\varphi((X_r)_{u\le r\le t})]
  \quad P^{s,x}\text{-a.s.}
$$
For bounded Borel $v:[0,\infty)\times E\to[0,\infty)$ and
$f\in B_b^+(E)$, the function
$$
  U^v_{s,t}f(x)
  :=\E^{s,x}\left[
     \exp\left(-\int_s^t v(r,X_r)\,dr\right)f(X_t)
  \right]
$$
is Borel measurable in $(s,t,x)$.
For $\mu\in\mathcal P(E)$, set
$P_\mu=\int_E P^{0,x}\,\mu(dx)$.
\end{assumption}

Assumption~\ref{ass:markov-class} is in force throughout the general
theory. After fixing $\mu$, we use the $P_\mu$-usual augmentation
$(\F_t)$ of $(\F_t^0)$ for adapted processes and randomized stopping
times. The Markov identities under $P_\mu$ remain valid after this
augmentation. This convention permits changes on a common
$P_\mu$-null set. Densities are identified up to
$dt\otimes P_\mu$-null sets; when defining Borel measures, we use jointly
Borel representatives of completed-measurable densities. All pathwise
identities are understood on a common set of full $P_\mu$-measure.
The assumptions allow general Borel transition kernels and unbounded
stopping intensities.

We first prove the time-inhomogeneous Duhamel formula used below;
see also \cite{Pazy1983} for the general semigroup context.

\begin{lemma}
\label{lem:duhamel}
For bounded Borel $v\ge0$, $f\in B_b^+(E)$, and $0\le s\le t$,
$$
  P_{s,t}f(x)
  =U^v_{s,t}f(x)
   +\int_s^t P_{s,u}\bigl(v(u,\cdot)U^v_{u,t}f\bigr)(x)\,du.
$$
\end{lemma}
\begin{proof}
For each path,
$$
  1-\exp\left(-\int_s^t v(r,X_r)\,dr\right)
  =\int_s^t v(u,X_u)
       \exp\left(-\int_u^t v(r,X_r)\,dr\right)\,du.
$$
Multiply by $f(X_t)$ and integrate under $P^{s,x}$. Tonelli's theorem
gives
$$
\begin{aligned}
  P_{s,t}f(x)-U^v_{s,t}f(x)
  =\int_s^t\E^{s,x}\left[
     v(u,X_u)\exp\left(-\int_u^t v(r,X_r)\,dr\right)f(X_t)
  \right]du.
\end{aligned}
$$
For fixed $u$, the factor after $v(u,X_u)$ depends only on the path on
$[u,t]$. Its conditional expectation given $\F_u^0$ equals
$U^v_{u,t}f(X_u)$ by Assumption~\ref{ass:markov-class}.
Since $v(u,X_u)$ is $\F_u^0$-measurable, the integrand equals
$$
  \E^{s,x}[v(u,X_u)U^v_{u,t}f(X_u)]
  =P_{s,u}\bigl(v(u,\cdot)U^v_{u,t}f\bigr)(x).
$$
This proves the formula.
\end{proof}

\begin{definition}
\label{def:rst}
A \emph{randomized stopping time} relative to $P_\mu$ is a probability
measure $\xi$ on $\Omega\times[0,\infty)$ whose first marginal is
$P_\mu$ and which has a disintegration
$$
  \xi(d\omega,dt)=\xi^\omega(dt)P_\mu(d\omega)
$$
such that $S_t^\xi(\omega):=1-\xi^\omega([0,t])$ is
$\F_t$-measurable for every $t\ge0$.
The class of these measures is denoted by $\RST$.
The \emph{observed measure} is $\pi_\xi:=(t,X_t)_\#\xi$ on
$[0,\infty)\times E$.
\end{definition}

The kernel $\xi^\omega$ is the conditional law of the stopping time given
the full path. Adaptedness of $S_t^\xi$ means that the surviving mass at
time $t$ is determined by the information available at that time.
All stopping times in $\RST$ are finite almost surely.

We need a survival identity that also covers non-locally integrable
intensities.

\begin{lemma}
\label{lem:deterministic-survival-identity}
Let $a:[0,\infty)\to[0,\infty)$ be a finite-valued measurable function,
and set
$$
  A_t=\int_0^t a_s\,ds,\qquad S_t=e^{-A_t},\qquad q_t=S_ta_t,
$$
with $e^{-\infty}=0$.
If $a\in L^1_{\mathrm{loc}}$, then for every $t\ge0$,
$$
  S_t=1-\int_0^t q_s\,ds.
$$
For arbitrary finite-valued measurable $a$, one always has
$\int_0^\infty q_s\,ds\le1$. If this integral equals $1$, the survival
identity holds for every $t\ge0$.
\end{lemma}
\begin{proof}
In the locally integrable case, $A$ and $S$ are absolutely continuous on
each finite interval and $S'_t=-S_ta_t$ almost everywhere. Integration
gives the identity.

For general $a$, let $a^n=a\wedge n$, $S_t^n=e^{-\int_0^t a_s^n ds}$,
and $q_t^n=S_t^na_t^n$. Then $S_t^n\downarrow S_t$ and
$q_t^n\to q_t$ for each $t$, because $a_t$ is finite.
The first part gives $\int_0^\infty q_t^n\,dt\le1$.
Fatou's lemma therefore gives $\int_0^\infty q_t\,dt\le1$.
If the latter integral equals $1$, Fatou's lemma also implies
$\int q^n\to1$. Since $\min(q^n,q)\to q$ and $\min(q^n,q)\le q$,
$$
  \int_0^\infty|q_t^n-q_t|\,dt
  =\int_0^\infty q_t^n\,dt+1
    -2\int_0^\infty\min(q_t^n,q_t)\,dt\longrightarrow0.
$$
Passing to the limit in $S_t^n=1-\int_0^t q_s^n\,ds$ proves the result.
\end{proof}

\begin{definition}
\label{def:intensity-generated-rst}
Let $\alpha$ be a finite-valued nonnegative progressively measurable
process, and write
$$
  A_t^\alpha=\int_0^t\alpha_s\,ds,\qquad
  S_t^\alpha=e^{-A_t^\alpha},\qquad
  q_t^\alpha=S_t^\alpha\alpha_t.
$$
The value $A_t^\alpha=\infty$ is allowed, in which case $S_t^\alpha=0$.
We call $\alpha$ an \emph{admissible intensity} if
$$
  S_t^\alpha=1-\int_0^t q_s^\alpha\,ds
  \quad(t\ge0),\qquad
  \int_0^\infty q_t^\alpha\,dt=1
  \quad P_\mu\text{-a.s.}
$$
The corresponding intensity-generated randomized stopping time is
$$
  \xi_\alpha(d\omega,dt)=q_t^\alpha(\omega)\,dt\,P_\mu(d\omega).
$$
\end{definition}

\begin{definition}
\label{def:markov-rst}
A randomized stopping time $\xi\in\RST$ is called \emph{Markovian} if
there is a finite-valued Borel function
$\lambda:[0,\infty)\times E\to[0,\infty)$ such that
$\alpha_t=\lambda(t,X_t)$ is admissible and $\xi=\xi_\alpha$.
We then write $\xi=\xi_\lambda$.
\end{definition}

\begin{lemma}
\label{lem:intensity-generated-rst-correct}
If $\alpha$ is admissible, then $\xi_\alpha$ is well defined and belongs
to $\RST$. Moreover, for every $t\ge0$,
$$
  \xi_\alpha^\omega([0,t])
  =\int_0^t q_s^\alpha(\omega)\,ds
  =1-S_t^\alpha(\omega)
  \quad P_\mu\text{-a.s.}
$$
\end{lemma}
\begin{proof}
Progressive measurability implies joint measurability of $q^\alpha$.
The normalization gives a probability measure with first marginal
$P_\mu$. Its conditional time density is $q^\alpha$. The displayed
identity follows from admissibility. Since $A_t^\alpha$ and
$S_t^\alpha$ are $\F_t$-measurable, $\xi_\alpha$ belongs to $\RST$.
\end{proof}

\begin{definition}
\label{def:unstopped-mass-flow}
For an admissible intensity $\alpha$, define the \emph{surviving mass
measure} $L_t^\alpha$ on $E$, its time integral $\overline L^\alpha$
(the \emph{surviving occupation measure}), and the observed measure
$\pi_\alpha$ by
$$
  L_t^\alpha(f)=\E_{P_\mu}[S_t^\alpha f(X_t)],\qquad
  \overline L^\alpha(dt,dx)=dt\,L_t^\alpha(dx),
$$
$$
  \pi_\alpha(dt,A)
  =\E_{P_\mu}[q_t^\alpha\1_{\{X_t\in A\}}]\,dt,
  \qquad A\in\Borel(E).
$$
Thus $\pi_\alpha=(t,X_t)_\#\xi_\alpha$.
\end{definition}

\begin{lemma}
\label{lem:unstopped-mass-balance}
For every admissible progressively measurable intensity $\alpha$,
$t\ge0$, and $f\in B_b^+(E)$,
$$
  L_t^\alpha(f)
  =\mu P_{0,t}f
   -\int_{[0,t]\times E}P_{s,t}f(x)\,\pi_\alpha(ds,dx).
$$
\end{lemma}
\begin{proof}
Admissibility gives
$$
  \E_{P_\mu}[f(X_t)]
  =\E_{P_\mu}[S_t^\alpha f(X_t)]
   +\int_0^t\E_{P_\mu}[q_s^\alpha f(X_t)]\,ds.
$$
By the Markov property and $\F_s$-measurability of $q_s^\alpha$,
$$
  \int_0^t\E_{P_\mu}[q_s^\alpha f(X_t)]\,ds
  =\int_{[0,t]\times E}P_{s,t}f(x)\,\pi_\alpha(ds,dx).
$$
All integrals are finite, since $f$ is bounded and
$\int_0^\infty q_s^\alpha\,ds=1$ almost surely.
Finally, $\E_{P_\mu}[f(X_t)]=\mu P_{0,t}f$.
\end{proof}

\begin{lemma}
\label{lem:balance-uniqueness}
Let $\lambda:[0,\infty)\times E\to[0,\infty)$ be Borel measurable.
Let $t\mapsto L_t$ be a Borel kernel of finite nonnegative measures on
$E$. Suppose that, for each $T<\infty$,
$$
  \int_0^T L_s(\lambda(s,\cdot))\,ds<\infty,
$$
and that, for every $t\ge0$ and $f\in B_b^+(E)$,
$$
  L_t(f)=\mu P_{0,t}f
    -\int_0^t L_s\bigl(\lambda(s,\cdot)P_{s,t}f\bigr)\,ds.
$$
Then
$$
  L_t(f)=\E_{P_\mu}[S_t^\lambda f(X_t)],\qquad
  S_t^\lambda=\exp\left(-\int_0^t\lambda(s,X_s)\,ds\right),
$$
with $e^{-\infty}=0$.
\end{lemma}
\begin{proof}
Fix $t\ge0$ and $f\in B_b^+(E)$. Set $\lambda_n=\lambda\wedge n$ and
$$
  U^n_{s,t}f(x)
  =\E^{s,x}\left[
     \exp\left(-\int_s^t\lambda_n(u,X_u)\,du\right)f(X_t)
   \right].
$$
Lemma~\ref{lem:duhamel} gives
$$
  P_{s,t}f
  =U^n_{s,t}f
   +\int_s^t P_{s,u}\bigl(\lambda_n(u,\cdot)U^n_{u,t}f\bigr)\,du.
$$
The functions in this formula are Borel measurable by
Assumption~\ref{ass:markov-class} and the monotone class theorem for
Borel kernels. Substitution into the balance equation yields
$$
\begin{aligned}
  L_t(f)
  &=\mu U^n_{0,t}f
    +\int_0^t\mu P_{0,u}
        \bigl(\lambda_n(u,\cdot)U^n_{u,t}f\bigr)\,du\\
  &\quad-\int_0^t L_s\bigl(\lambda(s,\cdot)U^n_{s,t}f\bigr)\,ds\\
  &\quad-\int_0^t\int_s^t
       L_s\Bigl(\lambda(s,\cdot)P_{s,u}
          \bigl(\lambda_n(u,\cdot)U^n_{u,t}f\bigr)\Bigr)\,du\,ds.
\end{aligned}
$$
Put $h_u^n(x)=\lambda_n(u,x)U^n_{u,t}f(x)$. Since
$0\le h_u^n\le n\|f\|_\infty$,
$$
  L_s\bigl(\lambda(s,\cdot)P_{s,u}h_u^n\bigr)
  \le n\|f\|_\infty L_s(\lambda(s,\cdot)).
$$
The right-hand side is integrable over the triangle $0\le s\le u\le t$.
We may therefore interchange the order of integration. Applying the
balance equation at time $u$ to $h_u^n$, we obtain
$$
  \mu P_{0,u}h_u^n
  =L_u(h_u^n)+\int_0^u
       L_s\bigl(\lambda(s,\cdot)P_{s,u}h_u^n\bigr)\,ds.
$$
Integrating in $u$ cancels the double integrals and gives
$$
  L_t(f)=\mu U^n_{0,t}f
    -\int_0^t L_s\bigl((\lambda-\lambda_n)(s,\cdot)
                          U^n_{s,t}f\bigr)\,ds.
$$
Dominated convergence under $P^{0,x}$ and then under $\mu$ gives
$$
  \mu U^n_{0,t}f\longrightarrow
  \E_{P_\mu}[S_t^\lambda f(X_t)].
$$
Also,
$$
  0\le L_s\bigl((\lambda-\lambda_n)(s,\cdot)U^n_{s,t}f\bigr)
   \le\|f\|_\infty L_s((\lambda-\lambda_n)(s,\cdot))
    \le\|f\|_\infty L_s(\lambda(s,\cdot)).
$$
For almost every $s$, $L_s$ integrates $\lambda(s,\cdot)$, so the
middle expression tends to zero. Dominated convergence in $s$ now
shows that the remaining integral vanishes in the limit.
\end{proof}

\begin{assumption}[Reference intensity]
\label{ass:reference-rate}
The Borel function $r:[0,\infty)\times E\to(0,\infty)$ satisfies
$$
  \int_0^T r(t,X_t)\,dt<\infty
  \quad P_\mu\text{-a.s. for every }T<\infty,
  \qquad
  \int_0^\infty r(t,X_t)\,dt=\infty
  \quad P_\mu\text{-a.s.}
$$
\end{assumption}

This assumption is imposed on all entropy statements below.
It implies that $r(t,X_t)$ is admissible and defines $\xi_r$.
For finite $t$, its survival function is strictly positive.

For probability measures $Q$ and $R$, we use
$$
  \KL(Q\|R)=\int\log\frac{dQ}{dR}\,dQ
$$
when $Q\ll R$, and set $\KL(Q\|R)=\infty$ otherwise. Equivalently,
if $h(u)=u\log u-u+1$, then
$\KL(Q\|R)=\int h(dQ/dR)\,dR$ when $Q\ll R$.
Here and below $0\log0=0$.
For $a>0$ and $z\ge0$, define
$$
  \ell_a(z)=z\log\frac za-z+a.
$$
Set $\ell_0(0)=0$ and $\ell_0(z)=\infty$ for $z>0$.
An expression $S\ell$ is interpreted as the integral of $\ell$ against
the measure with density $S$; in particular, its value is zero where
$S=0$, including when $\ell=\infty$.

\begin{proposition}
\label{prop:entropy-intensity}
Under Assumption~\ref{ass:reference-rate}, every admissible
progressively measurable intensity $\alpha$ satisfies
$$
  \KL(\xi_\alpha\|\xi_r)
  =\E_{P_\mu}\int_0^\infty
       S_t^\alpha\ell_{r(t,X_t)}(\alpha_t)\,dt.
$$
The equality holds in $[0,\infty]$.
\end{proposition}
\begin{proof}
We first work on a fixed path on which admissibility and the reference
conditions hold. Write $a_t=\alpha_t$, $b_t=r(t,X_t)$,
$S_t=e^{-\int_0^t a_s ds}$, $R_t=e^{-\int_0^t b_s ds}$,
$q_t=S_ta_t$, and $k_t=R_tb_t$.
Both $q_t\,dt$ and $k_t\,dt$ are probability measures, and $k_t>0$.
Let $H$ be their relative entropy.

For $T<\infty$, retain the exact stopping time on $[0,T]$ and put all
the mass after $T$ at a single additional point. The relative entropy
of the resulting two probability measures is
$$
  H_T=\int_0^T q_t\log\frac{q_t}{k_t}\,dt
        +S_T\log\frac{S_T}{R_T}.
$$
All terms are well defined as extended integrals. Put
$C_t=\int_0^t b_s\,ds$. Since $C_T<\infty$, the survival identity and
Tonelli's theorem give
$$
  \int_0^T q_t\log S_t\,dt+S_T\log S_T
  =-\int_0^T q_t\,dt,
\qquad
  \int_0^T q_t C_t\,dt+S_TC_T
  =\int_0^T S_tb_t\,dt.
$$
For the first identity, apply the chain rule to
$S\log S-S$ up to levels where $S$ is positive and then let the level
decrease to zero; the function has a finite limit at zero.
In particular, $\int_0^T q_t|\log S_t|\,dt\le1$.
Moreover,
$$
  q_t\log\frac{a_t}{b_t}\ge-\frac{S_tb_t}{e},
$$
so its negative part is integrable on $[0,T]$.
Hence the logarithmic terms can be combined, and the two identities yield
$$
  H_T=\int_0^T S_t\ell_{b_t}(a_t)\,dt.
$$

To pass to infinite time, use the nonnegative function
$h(u)=u\log u-u+1$. We have
$$
  H=\int_0^\infty k_t h(q_t/k_t)\,dt,
  \qquad
  H_T=\int_0^T k_t h(q_t/k_t)\,dt
          +R_T h(S_T/R_T).
$$
Jensen's inequality on the tail $(T,\infty)$ gives $H_T\le H$,
whereas nonnegativity gives
$H_T\ge\int_0^T k_t h(q_t/k_t)\,dt$.
It follows that $H_T\to H$, including when $H=\infty$.
Monotone convergence on the intensity side proves the pathwise identity.

Finally, the conditional measures of $\xi_\alpha$ and $\xi_r$ have the
densities just considered, with the same first marginal $P_\mu$.
The chain rule for relative entropy \cite{DupuisEllis} gives
$$
  \KL(\xi_\alpha\|\xi_r)
   =\int_\Omega\KL(\xi_\alpha^\omega\|\xi_r^\omega)\,P_\mu(d\omega).
$$
In this density setting the chain rule follows directly from Tonelli's
theorem applied to the nonnegative integrand $k_t h(q_t/k_t)$.
Integrating the pathwise identity proves the proposition.
\end{proof}

\begin{lemma}
\label{lem:progressive-density}
Let $A$ be an adapted continuous nondecreasing process with $A_0=0$.
Suppose that, for a jointly measurable $a\ge0$,
$$
  A_t=\int_0^t a_s\,ds\qquad(t\ge0)
$$
almost surely. Then $a$ has a finite-valued progressively measurable
representative with respect to $dt\otimes P_\mu$.
\end{lemma}
\begin{proof}
Set $a_t^n=n(A_t-A_{(t-1/n)^+})$.
Each $a^n$ is adapted and continuous, hence progressively measurable.
Their pointwise upper limit is progressively measurable and, by the
Lebesgue differentiation theorem, equals $a$ for
$dt\otimes P_\mu$-almost every $(t,\omega)$.
The integrability of $a$ on finite intervals implies that this upper
limit is finite almost everywhere. Set it equal to zero wherever it
is infinite. This gives the required representative.
\end{proof}

\section{Markovization and its properties}
\label{sec:markovization}

We first construct a Markovian representative using the Markov
structure and the measure-valued balance relation. Under
Assumption~\ref{ass:reference-rate}, we then derive the entropy inequality
and its exact remainder formula.

\begin{theorem}
\label{thm:markovization-intensity}
Let $\alpha$ be an admissible progressively measurable intensity, and
let $L^\alpha$, $\overline L^\alpha$, and $\pi_\alpha$ be as in
Definition~\ref{def:unstopped-mass-flow}.
Then $\pi_\alpha\ll\overline L^\alpha$.
Every finite-valued Borel version
$$
  \lambda^\alpha(t,x)=\frac{d\pi_\alpha}{d\overline L^\alpha}(t,x)
$$
defines an admissible Markovian intensity. Moreover,
$$
  L_t^{\lambda^\alpha}=L_t^\alpha\quad(t\ge0),
  \qquad
  \pi_{\xi_{\lambda^\alpha}}=\pi_\alpha.
$$
\end{theorem}
\begin{proof}
If a Borel set $A\subset[0,\infty)\times E$ satisfies
$\overline L^\alpha(A)=0$, then
$S_t^\alpha\1_A(t,X_t)=0$ almost everywhere under $dt\otimes P_\mu$.
Since $\alpha$ is finite-valued,
$q_t^\alpha\1_A(t,X_t)=0$ almost everywhere as well. Thus
$\pi_\alpha(A)=0$.
The measure $\overline L^\alpha$ is $\sigma$-finite, since
$\overline L^\alpha([0,T]\times E)\le T$.
As $\pi_\alpha$ is finite, its Radon--Nikodym derivative is finite
$\overline L^\alpha$-almost everywhere. Choose a finite-valued Borel
version, denoted by $\lambda^\alpha$.

Joint measurability of $S^\alpha$ and $X$ implies that $t\mapsto
L_t^\alpha$ is a Borel kernel. The process
$\lambda^\alpha(t,X_t)$ is progressively measurable because $X$ is
adapted and has c\`adl\`ag paths. By
Lemma~\ref{lem:unstopped-mass-balance},
$$
  L_t^\alpha(f)=\mu P_{0,t}f
     -\int_0^t L_s^\alpha\bigl(
          \lambda^\alpha(s,\cdot)P_{s,t}f\bigr)\,ds.
$$
For every $T<\infty$,
$$
  \int_0^T L_s^\alpha(\lambda^\alpha(s,\cdot))\,ds
  =\pi_\alpha([0,T]\times E)\le1.
$$
Lemma~\ref{lem:balance-uniqueness} therefore gives, for every $t$ and
$f\in B_b^+(E)$,
$$
  L_t^\alpha(f)
  =\E_{P_\mu}\left[
       \exp\left(-\int_0^t\lambda^\alpha(s,X_s)\,ds\right)f(X_t)
    \right].
$$
Set
$$
  \widetilde S_t
  =\exp\left(-\int_0^t\lambda^\alpha(s,X_s)\,ds\right),
  \qquad
  \widetilde L_t(f)=\E_{P_\mu}[\widetilde S_t f(X_t)],
  \qquad
  \widetilde q_t=\widetilde S_t\lambda^\alpha(t,X_t).
$$
The preceding identity gives $\widetilde L_t=L_t^\alpha$.
Tonelli's theorem then yields
$$
  \E_{P_\mu}\int_0^\infty \widetilde q_t\,dt
  =\int_0^\infty
       \widetilde L_t(\lambda^\alpha(t,\cdot))\,dt
  =\int_0^\infty L_t^\alpha(\lambda^\alpha(t,\cdot))\,dt
   =\pi_\alpha([0,\infty)\times E)=1.
$$
Lemma~\ref{lem:deterministic-survival-identity}, applied pathwise, gives
$\int_0^\infty \widetilde q_t\,dt\le1$.
Since its expectation is $1$, this integral equals $1$ almost surely.
The same lemma gives
$$
  \widetilde S_t
  =1-\int_0^t \widetilde q_s\,ds
  \qquad(t\ge0)
$$
on a common set of full probability. Thus $\lambda^\alpha$ is admissible,
$\widetilde S=S^{\lambda^\alpha}$, and
$\widetilde L_t=L_t^{\lambda^\alpha}=L_t^\alpha$.

Finally, for every Borel $A\subset[0,\infty)\times E$,
$$
  \pi_{\xi_{\lambda^\alpha}}(A)
  =\E_{P_\mu}\int_0^\infty
       \widetilde q_t\1_A(t,X_t)\,dt
  =\int_A\lambda^\alpha(t,x)L_t^\alpha(dx)\,dt
   =\pi_\alpha(A).
$$
\end{proof}

Finite relative entropy also provides an intensity representation.

\begin{proposition}
\label{prop:finite-entropy-intensity}
Under Assumption~\ref{ass:reference-rate}, if $\xi\in\RST$ and
$\KL(\xi\|\xi_r)<\infty$, there is an admissible progressively
measurable intensity $\alpha$ such that $\xi=\xi_\alpha$.
\end{proposition}
\begin{proof}
Finite entropy implies $\xi\ll\xi_r$.
Since both measures have first marginal $P_\mu$, disintegration gives
$\xi^\omega\ll\xi_r^\omega$ for $P_\mu$-almost every $\omega$.
The reference conditional law has strictly positive density
$S_t^r r(t,X_t)$ with respect to $dt$ for finite $t$.
Hence $\xi^\omega\ll dt$ almost surely.
The Radon--Nikodym theorem for kernels gives a jointly measurable
density $a\ge0$ such that
$$
  \xi^\omega(dt)=a_t(\omega)\,dt
  \quad P_\mu\text{-a.s.}
$$
Disintegration and the kernel density theorem apply because the
underlying Borel spaces are standard Borel; see
\cite{BogachevMeasureTheory}.

The process $A_t^\xi=\xi^\omega([0,t])$ is adapted and has almost surely
absolutely continuous paths, with
$A_t^\xi=\int_0^t a_s\,ds$ and $A_0^\xi=0$.
Using the completed filtration and
Lemma~\ref{lem:progressive-density}, choose $a$ finite-valued and
progressively measurable. Write $S_t^\xi=1-A_t^\xi$ and set
$$
  \alpha_t=
  \begin{cases}
     a_t/S_t^\xi,&S_t^\xi>0,\\
     0,&S_t^\xi=0.
  \end{cases}
$$
This is a finite-valued progressively measurable process.

Fix a path on which these identities hold, and let
$\sigma=\inf\{t:S_t^\xi=0\}$, with $\inf\varnothing=\infty$.
For $t<\sigma$, $S_s^\xi\ge S_t^\xi>0$ on $[0,t]$. Thus $\alpha$ is
integrable there, and the chain rule gives
$$
  -\log S_t^\xi=\int_0^t\alpha_s\,ds.
$$
If $\sigma<\infty$, continuity gives $S_t^\xi\downarrow0$ as
$t\uparrow\sigma$, so the integral diverges at $\sigma$.
The same exponential representation therefore holds for all $t$, with
$e^{-\infty}=0$. Since a nonnegative nonincreasing survival function
stays zero after its first zero, $a_t=0$ for almost every $t$ on
$\{S_t^\xi=0\}$. Consequently,
$$
  a_t=S_t^\xi\alpha_t\quad dt\otimes P_\mu\text{-a.e.},
  \qquad
  \int_0^\infty S_t^\xi\alpha_t\,dt=1
  \quad P_\mu\text{-a.s.}
$$
The original survival identity now proves admissibility and
$\xi=\xi_\alpha$.
\end{proof}

\begin{theorem}
\label{thm:entropy-markovization}
Assume Assumption~\ref{ass:reference-rate}.
Let $\alpha$ be admissible, with $\KL(\xi_\alpha\|\xi_r)<\infty$,
and let $\lambda^\alpha$ be constructed in
Theorem~\ref{thm:markovization-intensity}. Then
$$
  \KL(\xi_\alpha\|\xi_r)
  =\KL(\xi_{\lambda^\alpha}\|\xi_r)
  +\E_{P_\mu}\int_0^\infty
       S_t^\alpha\ell_{\lambda^\alpha(t,X_t)}(\alpha_t)\,dt.
$$
In particular, $\KL(\xi_{\lambda^\alpha}\|\xi_r)
\le\KL(\xi_\alpha\|\xi_r)$.
Equality holds if and only if
$\alpha_t=\lambda^\alpha(t,X_t)$ almost everywhere with respect to
$S_t^\alpha\,dt\,P_\mu(d\omega)$.
\end{theorem}
\begin{proof}
Define
$$
  m(d\omega,dt)=S_t^\alpha(\omega)\,dt\,P_\mu(d\omega),
  \qquad \Phi(\omega,t)=(t,X_t(\omega)).
$$
Then $\Phi_\#m=\overline L^\alpha$ and
$\Phi_\#(\alpha m)=\pi_\alpha$.
The definition of $\lambda^\alpha$ implies that, for every bounded
Borel function $\varphi$ on time--state space,
$$
  \int\varphi(t,X_t)\alpha_t\,dm
  =\int\varphi(t,X_t)\lambda^\alpha(t,X_t)\,dm.
$$
Write $z=\alpha_t$, $y=\lambda^\alpha(t,X_t)$, and $b=r(t,X_t)$.
Both $z$ and $y$ are integrable under $m$, with integral $1$.
Taking $\varphi=\1_{\{\lambda^\alpha=0\}}$ shows that $z=0$
$m$-almost everywhere on $\{y=0\}$.
On this set,
$$
  \ell_b(z)=\ell_b(y)=b,\qquad \ell_y(z)=0.
$$

On $\{y>0\}$, the algebraic identity
$$
  \ell_b(z)=\ell_b(y)+\ell_y(z)+(z-y)\log\frac yb
$$
holds. For $N\ge1$, put
$$
  A_N=\left\{y>0,\ \left|\log\frac yb\right|\le N\right\},
  \qquad g_N=\1_{A_N}\log\frac yb.
$$
The function $g_N$ is bounded and is a Borel function of $(t,X_t)$.
The defining identity for $\lambda^\alpha$ yields
$\int(z-y)g_N\,dm=0$.
Since $\int\ell_b(z)\,dm<\infty$ by
Proposition~\ref{prop:entropy-intensity}, integration over $A_N$ is
justified and gives
$$
  \int_{A_N}\ell_b(z)\,dm
  =\int_{A_N}\ell_b(y)\,dm+\int_{A_N}\ell_y(z)\,dm.
$$
The three integrands are nonnegative, and
$A_N\uparrow\{y>0\}$. Monotone convergence, followed by adding the
already verified identity on $\{y=0\}$, gives
$$
  \int\ell_b(z)\,dm
  =\int\ell_b(y)\,dm+\int\ell_y(z)\,dm.
$$
By Proposition~\ref{prop:entropy-intensity}, the left-hand side is
$\KL(\xi_\alpha\|\xi_r)$.
Theorem~\ref{thm:markovization-intensity} gives
$L_t^{\lambda^\alpha}=L_t^\alpha$, hence
$$
  \int\ell_{r(t,X_t)}(\lambda^\alpha(t,X_t))\,dm
  =\int_0^\infty L_t^\alpha
        \bigl(\ell_{r(t,\cdot)}(\lambda^\alpha(t,\cdot))\bigr)\,dt
  =\int_0^\infty L_t^{\lambda^\alpha}
        \bigl(\ell_{r(t,\cdot)}(\lambda^\alpha(t,\cdot))\bigr)\,dt
  =\KL(\xi_{\lambda^\alpha}\|\xi_r).
$$
This proves the decomposition.
For every $a\ge0$, $\ell_a(z)\ge0$, with equality exactly when $z=a$,
including the convention for $a=0$.
This proves the inequality and its equality criterion.
\end{proof}

\section{Reduction of variational problems}
\label{sec:variational}

We begin with an example showing how the stopped-path law can change
under Markovization. We then prove the canonical minimizer and the
reduction of entropy-penalized problems.

\begin{example}
\label{ex:markovization-limit}
Let $E=\{-1,0,1\}$ and let $\mu=\tfrac12(\delta_{-1}+\delta_1)$.
Let $X_t=X_0$ for $t<1$ and $X_t=0$ for $t\ge1$.
This is a time-inhomogeneous c\`adl\`ag Markov process.
Take $\alpha_t=0$ before time $1$. After time $1$, take rate $2$ on
$\{X_{1-}=1\}$ and rate $1$ on $\{X_{1-}=-1\}$.
The observed measure is supported on $[1,\infty)\times\{0\}$.

Writing $u=t-1>0$, the Markovian rate on this support is
$$
  \lambda^\alpha(t,0)
  =\frac{2e^{-2u}+e^{-u}}{e^{-2u}+e^{-u}}.
$$
It is zero before time $1$. Its survival function after time $1$ is
$(e^{-2u}+e^{-u})/2$ on both initial branches.
Thus, for the original rule and its Markovization, respectively,
$$
  \xi_\alpha\bigl(\{X_{1-}=1,\tau>t\}\bigr)=\tfrac12e^{-2u},
  \qquad
  \xi_{\lambda^\alpha}\bigl(\{X_{1-}=1,\tau>t\}\bigr)
      =\tfrac14(e^{-2u}+e^{-u}).
$$
Here $\tau$ denotes the time coordinate on $\Omega\times[0,\infty)$.
The two quantities differ, even though the observed measures agree.
For example, their difference changes the expectation of the functional
$\1_{\{\tau>t\}}\1_{\{X_{1-}=1\}}$, with value zero when $\tau\le t$;
this is a functional of the stopped path together with its lifetime.
The unmarked trajectory $(X_{s\wedge\tau})_{s\ge0}$ is unchanged in
this particular example, because $X$ is already constant after time $1$.

The following modification also changes the law of the unmarked stopped
trajectory. Take $E=\R$ and set $X_t=t-1$ for $t\ge1$, keeping the same
initial branches and intensities.
The two displayed probabilities remain unchanged. For $t>1$, the event
$\{X_{1-}=1,\tau>t\}$ agrees almost surely with
$\{X_0=1,X_{t\wedge\tau}=t-1\}$, since the stopping laws have no
time atoms. Their observed measures agree. The joint laws of $X_0$ and
$X_{t\wedge\tau}$ are different.
\end{example}

\begin{corollary}
\label{cor:canonical-pi}
Assume Assumption~\ref{ass:reference-rate}.
Suppose that $\pi$ is the observed measure of some $\xi\in\RST$ with
$\KL(\xi\|\xi_r)<\infty$. Set
$$
  \mathcal A(\pi)=\{\zeta\in\RST:\pi_\zeta=\pi\}.
$$
Then $\KL(\cdot\|\xi_r)$ has a unique minimizer on $\mathcal A(\pi)$.
It is the Markovian stopping time $\xi_{\lambda^\pi}$, where
$$
  \lambda^\pi=\frac{d\pi}{d\overline L^\pi},\qquad
  \overline L^\pi(dt,dx)=dt\,L_t^\pi(dx),
$$
$$
  L_t^\pi(f)=\mu P_{0,t}f
     -\int_{[0,t]\times E}P_{s,t}f(x)\,\pi(ds,dx),
  \qquad f\in B_b^+(E).
$$
\end{corollary}
\begin{proof}
By Proposition~\ref{prop:finite-entropy-intensity}, the given $\xi$ has
an admissible intensity $\alpha$.
Lemma~\ref{lem:unstopped-mass-balance} identifies $L_t^\alpha$ with
$L_t^\pi$. In particular, the displayed formula defines a nonnegative
measure and a Borel kernel, and $\pi\ll\overline L^\pi$.
Fix one finite-valued Borel version $\lambda^\pi$.
Theorems~\ref{thm:markovization-intensity}
and~\ref{thm:entropy-markovization} construct
$\xi_{\lambda^\pi}\in\mathcal A(\pi)$ with finite entropy.

Now let $\zeta\in\mathcal A(\pi)$ be any competitor with finite
entropy. It has an admissible intensity $\beta$ and the same surviving
mass measures $L_t^\beta=L_t^\pi$.
Thus the same function $\lambda^\pi$ is a valid Radon--Nikodym version
for its Markovization. The entropy inequality gives
$$
  \KL(\xi_{\lambda^\pi}\|\xi_r)\le\KL(\zeta\|\xi_r).
$$
Hence $\xi_{\lambda^\pi}$ minimizes the relative entropy on
$\mathcal A(\pi)$.

The set $\mathcal A(\pi)$ is convex: both marginals, the observed
measure, and adaptedness of the conditional distribution functions are
preserved by mixtures. On measures of finite entropy,
$\KL(\cdot\|\xi_r)$ is strictly convex, because $u\mapsto u\log u$
is strictly convex on $[0,\infty)$.
Two distinct minimizers would therefore have a midpoint with strictly
smaller entropy. This proves uniqueness.
\end{proof}

\begin{corollary}
\label{cor:reduction-to-markov}
Under Assumption~\ref{ass:reference-rate}, let
$\Psi:\mathcal P([0,\infty)\times E)\to(-\infty,\infty]$ be any
functional. Then
$$
  \inf_{\xi\in\RST}
       \{\Psi(\pi_\xi)+\KL(\xi\|\xi_r)\}
  =\inf_\lambda
       \{\Psi(\pi_{\xi_\lambda})+\KL(\xi_\lambda\|\xi_r)\},
$$
where the second infimum is over all admissible Borel Markovian
intensities. The equality includes the values $-\infty$ and $\infty$.
\end{corollary}
\begin{proof}
One inequality follows from inclusion of the Markovian class in $\RST$.
For every $\xi$ with finite objective value,
Proposition~\ref{prop:finite-entropy-intensity} and
Theorem~\ref{thm:entropy-markovization} give a Markovian rule with the
same observed measure and entropy at most $\KL(\xi\|\xi_r)$.
Taking infima gives the reverse inequality.
\end{proof}

\begin{remark}[Connection with optimal Skorokhod embedding]
\label{rem:skorokhod-connection}
In \cite{BeiglbockCoxHuesmann2017}, the optimal embedding problem is
formulated for functions of the stopped path.
If $\gamma((B_s)_{s\le t},t)=G(t,B_t)$, then
$$
  \int\gamma((B_s)_{s\le t},t)\,\xi(d\omega,dt)
  =\int_{[0,\infty)\times\R}G(t,x)\,\pi_\xi(dt,dx),
$$
whenever the integral is well defined.
Such objectives fit the reduction directly, provided that all additional
admissibility constraints are also preserved.

For example, the uniform-integrability condition in Brownian embedding
is preserved. Indeed, if $\tau$ is the time coordinate, then for every
$t\ge0$,
$$
  (X_{t\wedge\tau})_\#\xi_\alpha(dx)
  =L_t^\alpha(dx)+\pi_\alpha([0,t],dx).
$$
This follows by splitting into $\{\tau>t\}$ and $\{\tau\le t\}$.
Both terms are preserved by Markovization. Thus all one-time laws of
the stopped process are preserved. Their joint laws may change. For
real-valued processes, uniform integrability of the family
$(X_{t\wedge\tau})_{t\ge0}$ depends only on these one-time laws and is
therefore preserved as well.

Some other path-dependent criteria can be included by enlarging the
state. For example, let
$Y_t=(B_t,\overline B_t)$, where
$\overline B_t=\sup_{s\le t}B_s$.
This is a continuous Markov process on the Polish state space
$\{(x,m)\in\R^2:x\le m\}$, with initial law given by
$x\mapsto(x,x)$ under $\mu$.
If $\gamma((B_s)_{s\le t},t)=G(t,B_t,\overline B_t)$, then
$$
  \int\gamma((B_s)_{s\le t},t)\,\xi(d\omega,dt)
  =\int G(t,x,m)\,(t,Y_t)_\#\xi(dt,dx,dm).
$$
Markovization must now be applied to $Y$, so the resulting intensity is
allowed to depend on the running maximum as well as on $B_t$.

Similarly, take an \emph{open} set $D\subset\R$, and put
$$
  \tau_D=\inf\{t\ge0:B_t\notin D\},\qquad
  I_t^D=\1_{\{\tau_D\le t\}},\qquad Y_t=(B_t,I_t^D).
$$
This records whether the path has left $D$.
The augmented process is c\`adl\`ag and Markovian. A canonical Markov
family on $\R\times\{0,1\}$ is obtained by using Brownian motion with
flag $1$ after exit, and the stated exit rule when starting from
$(x,0)$ with $x\in D$. The family is completed on the inaccessible states
$(x,0)$, $x\notin D$, by an absorbing extension. Its natural initial law is
$x\mapsto(x,\1_{\{x\notin D\}})$.
The Brownian construction and the Borel exit-time functional give the
measurability required in Assumption~\ref{ass:markov-class}.
For $\gamma((B_s)_{s\le t},t)=G(t,B_t,I_t^D)$,
$$
  \int\gamma((B_s)_{s\le t},t)\,\xi(d\omega,dt)
  =\int G(t,x,i)\,(t,Y_t)_\#\xi(dt,dx,di).
$$

Running minima, running absolute maxima, and suitable finite collections
of exit indicators can be handled in the same way. For a general path
statistic, this extension requires a Markovian augmentation satisfying
Assumption~\ref{ass:markov-class}. The corresponding observed measure is
$(t,Y_t)_\#\xi$ for the augmented state process $Y$.
\end{remark}

\section{Approximation by bounded Markovian intensities}
\label{sec:smoothing}

We next approximate the observed measure of any randomized stopping
time. First add an independent exponential delay. The resulting rule
has a bounded progressively measurable intensity. Its Markovization
has the same bound and the same observed measure.

\begin{lemma}
\label{lem:exp-smoothing}
Let $\xi\in\RST$. For $n\ge1$, define $\xi_n$ by its conditional
measures
$$
  \xi_n^\omega(dt)=f_t^n(\omega)\,dt,\qquad
  f_t^n(\omega)=\int_{[0,t)}ne^{-n(t-s)}\,\xi^\omega(ds).
$$
Then $\xi_n\in\RST$ and $\xi_n=\xi_{\alpha^n}$ for an admissible
progressively measurable intensity with $0\le\alpha^n\le n$.
If $\int t\,\xi(d\omega,dt)<\infty$, then
$$
  \int t\,\xi_n(d\omega,dt)
  =\int t\,\xi(d\omega,dt)+\frac1n.
$$
\end{lemma}
\begin{proof}
Tonelli's theorem gives $\int_0^\infty f_t^n\,dt=1$ for almost every
path, since
$$
  \int_0^\infty f_t^n\,dt
  =\int_{[0,\infty)}\int_s^\infty ne^{-n(t-s)}\,dt\,\xi^\omega(ds)=1.
$$
Write $A_t^\xi=\xi^\omega([0,t])$ and
$A_t^n=\xi_n^\omega([0,t])$. Another application of Tonelli gives
$$
  A_t^n
  =\int_{[0,t)}(1-e^{-n(t-s)})\,\xi^\omega(ds)
  =\int_0^t ne^{-n(t-s)}A_s^\xi\,ds
   =\int_0^t f_s^n\,ds.
$$
The process $A^\xi$ is adapted and right-continuous, hence progressive.
The displayed expression shows that $A^n$ is adapted and continuous,
with $A_0^n=0$. Thus $\xi_n\in\RST$.

Let $S_t^n=1-A_t^n$. Pathwise,
$$
  S_t^n=\xi^\omega([t,\infty))
       +\int_{[0,t)}e^{-n(t-s)}\,\xi^\omega(ds)\ge e^{-nt}>0,
$$
and $f_t^n\le nS_t^n$.
By Lemma~\ref{lem:progressive-density}, the density has a finite-valued
progressively measurable representative. We choose one satisfying
$0\le f_t^n\le nS_t^n$, which is possible by truncation at $nS_t^n$.
Set $\alpha_t^n=f_t^n/S_t^n$.
It is progressively measurable and bounded by $n$.
Since $dS_t^n=-S_t^n\alpha_t^n\,dt$ and $S_0^n=1$,
$$
  S_t^n=\exp\left(-\int_0^t\alpha_s^n\,ds\right).
$$
The normalization of $f^n$ proves admissibility.

Conditionally on $\omega$, $\xi_n^\omega$ is the law of $s+u$ under
$\xi^\omega(ds)\otimes ne^{-nu}\,du$.
Its first moment is therefore the first moment of $\xi^\omega$ plus
$1/n$. Integration under $P_\mu$ proves the last assertion.
\end{proof}

\begin{lemma}
\label{lem:smoothed-observed-measures}
For the measures in Lemma~\ref{lem:exp-smoothing},
$$
  \pi_{\xi_n}\Rightarrow\pi_\xi
  \quad\text{in }\mathcal P([0,\infty)\times E).
$$
\end{lemma}
\begin{proof}
Let $F\in C_b([0,\infty)\times E)$.
Using the conditional description of the exponential delay and the
change of variables $u\mapsto u/n$, we have
$$
  \int F\,d\pi_{\xi_n}
  =\int_{\Omega\times[0,\infty)\times[0,\infty)}
       F(t+u/n,X_{t+u/n}(\omega))\,
       \xi(d\omega,dt)e^{-u}\,du.
$$
For every path and every $t,u\ge0$, right continuity gives
$$
  F(t+u/n,X_{t+u/n}(\omega))\longrightarrow F(t,X_t(\omega)).
$$
Since $F$ is bounded, dominated convergence gives
$\int F\,d\pi_{\xi_n}\to\int F\,d\pi_\xi$.
\end{proof}

\begin{theorem}
\label{thm:bounded-markov-approximation}
For every $\xi\in\RST$, there are Borel functions
$\lambda^n:[0,\infty)\times E\to[0,\infty)$ with
$0\le\lambda^n\le n$, each defining an admissible Markovian intensity,
such that
$$
  \pi_{\xi_{\lambda^n}}\Rightarrow\pi_\xi
  \quad\text{in }\mathcal P([0,\infty)\times E).
$$
If $\int t\,\xi(d\omega,dt)<\infty$, then
$$
  \int t\,\pi_{\xi_{\lambda^n}}(dt,dx)
  =\int t\,\pi_\xi(dt,dx)+\frac1n.
$$
\end{theorem}
\begin{proof}
By Lemma~\ref{lem:exp-smoothing}, $\xi_n=\xi_{\alpha^n}$ with
$0\le\alpha^n\le n$.
For every Borel $A\subset[0,\infty)\times E$,
$$
  \pi_{\xi_n}(A)
  =\E_{P_\mu}\int_0^\infty
      \1_A(t,X_t)S_t^{\alpha^n}\alpha_t^n\,dt
  \le n\,\overline L^{\alpha^n}(A).
$$
Hence the derivative
$\lambda^n=d\pi_{\xi_n}/d\overline L^{\alpha^n}$ has a Borel version
with values in $[0,n]$.
Theorem~\ref{thm:markovization-intensity} shows that this version is
admissible and that $\pi_{\xi_{\lambda^n}}=\pi_{\xi_n}$.
Lemma~\ref{lem:smoothed-observed-measures} gives the weak convergence.
Finally, Markovization preserves the time marginal of the observed
measure, so the first-moment identity follows from
Lemma~\ref{lem:exp-smoothing}.
\end{proof}

Thus the observed measures generated by bounded Markovian intensities
are weakly dense among the observed measures of randomized stopping
times. If $(X_t)_\#\xi=\nu$, the corresponding terminal laws converge
weakly to $\nu$, since $(t,x)\mapsto x$ is continuous.

\section{Finite-dimensional constraints in Skorokhod embedding}
\label{sec:brownian}

We now specialize to one-dimensional Brownian motion and consider
finitely many linear constraints on its terminal law. The prescribed
vector is taken from the relative interior of a convex-order moment
region.

Let $B_t=B_0+W_t$, where $W$ is standard Brownian motion, independent of
$B_0$, and $B_0$ has law $\mu\in\mathcal P_1(\R)$.
Here $\mathcal P_1(\R)$ consists of probability measures with finite
first absolute moment. For $\xi\in\RST$, write
$p_\xi=(\operatorname{pr}_2)_\#\pi_\xi$ for its terminal law, where
$\operatorname{pr}_2(t,x)=x$.
Let $g_1,\ldots,g_m:\R\to\R$ be continuous functions of at most
linear growth, and set
$$
  \mathcal C_\mu(g_1,\ldots,g_m)
  =\left\{
    \left(\int_\R g_i\,d\eta\right)_{i=1}^m:
    \eta\in\mathcal P_1(\R),\ \mu\cx\eta
   \right\}.
$$
Convex order means that $\int\varphi\,d\mu\le\int\varphi\,d\eta$
for every real-valued convex function $\varphi$, with the usual
extended-integral interpretation. In particular, the means agree.
The set $\mathcal C_\mu$ is nonempty and convex: it contains the vector
for $\eta=\mu$, and mixtures preserve convex order and the moment map.

We use the following external embedding result: if
$\mu,\eta\in\mathcal P_1(\R)$ and $\mu\cx\eta$, there is an almost
surely finite ordinary Brownian stopping time $\tau$ such that
$B_\tau\sim\eta$ and $(B_{t\wedge\tau})_{t\ge0}$ is uniformly
integrable. For a precise version with a general initial distribution,
see \cite[Theorem~2.8]{CoxOblojTouzi2019}; a translation by the common
mean removes the centering convention there. The case $\mu=\eta$ is
realized by $\tau=0$.

\begin{lemma}
\label{lem:brownian-w1}
Let $\tau$ be an almost surely finite stopping time with
$B_\tau\sim\eta\in\mathcal P_1(\R)$.
Let $e_n\sim\operatorname{Exp}(n)$ be independent of $(B,\tau)$.
Then $\eta_n=\Law(B_{\tau+e_n})$ belongs to $\mathcal P_1(\R)$ and
$$
  W_1(\eta_n,\eta)\le\frac1{\sqrt{2n}}\longrightarrow0.
$$
\end{lemma}
\begin{proof}
The pair $(B_{\tau+e_n},B_\tau)$ is a coupling of $\eta_n$ and $\eta$.
By the strong Markov property, the increments after $\tau$ form an
independent Brownian motion. Conditioning on $e_n$ gives
$$
  \E|B_{\tau+e_n}-B_\tau|
  =\sqrt{\frac2\pi}\,\E\sqrt{e_n}
  =\sqrt{\frac2\pi}\,\frac{\Gamma(3/2)}{\sqrt n}
   =\frac1{\sqrt{2n}}.
$$
This proves the bound and, together with $\E|B_\tau|<\infty$, the
finite first moment of $\eta_n$.
\end{proof}

\begin{lemma}
\label{lem:ri-conv}
Let $C\subset\R^m$ be convex and let $D\subset C$ be dense in $C$
in the relative topology. Then $\ri C\subset\conv D$.
\end{lemma}
\begin{proof}
There is nothing to prove if $C$ is empty.
Let $c\in\ri C$ and $d=\dim\aff C$.
If $d=0$, then $C=\{c\}$ and density gives $D=C$.
For $d>0$, a sufficiently small ball about $c$ in $\aff C$ lies in $C$.
Choose a nondegenerate $d$-dimensional simplex inside this ball with
vertices $a^0,\ldots,a^d$ and with $c$ in its relative interior.
Density allows each vertex to be replaced by a sufficiently close point
$b^j\in D$.
The barycentric coordinates of $c$ depend continuously on the vertices
near a nondegenerate simplex. They therefore remain strictly positive
for sufficiently small perturbations. Thus
$c\in\conv\{b^0,\ldots,b^d\}\subset\conv D$.
\end{proof}

\begin{proposition}
\label{prop:finite-terminal-constraints}
Suppose that $g_1,\ldots,g_m$ are continuous and have at most linear
growth. If
$$
  c\in\ri\mathcal C_\mu(g_1,\ldots,g_m),
$$
then there is an admissible bounded Borel Markovian intensity
$\lambda:[0,\infty)\times\R\to[0,\infty)$ such that
$p_{\xi_\lambda}\in\mathcal P_1(\R)$, $\mu\cx p_{\xi_\lambda}$, and
$$
  \int_\R g_i(x)\,p_{\xi_\lambda}(dx)=c_i,
  \qquad i=1,\ldots,m.
$$
\end{proposition}
\begin{proof}
Let $D_0$ be the set of vectors obtained by the following procedure.
Choose $\eta\in\mathcal P_1(\R)$ with $\mu\cx\eta$, choose an
embedding $B_\tau\sim\eta$, add an independent delay of law
$\operatorname{Exp}(n)$ for some integer $n\ge1$, and Markovize the
resulting randomized stopping time.
By Theorem~\ref{thm:bounded-markov-approximation}, its intensity can be
chosen bounded by $n$.

To check $D_0\subset\mathcal C_\mu$, set
$\eta_n=\Law(B_{\tau+e_n})$.
Lemma~\ref{lem:brownian-w1} gives $\eta_n\in\mathcal P_1(\R)$.
The strong Markov property and independence of $e_n$ imply
$$
  \E[B_{\tau+e_n}\mid B_\tau]=B_\tau.
$$
Conditional Jensen's inequality gives $\eta\cx\eta_n$, hence
$\mu\cx\eta_n$.
Markovization preserves this terminal law, so its moment vector is in
$\mathcal C_\mu$.

Now take any $a\in\mathcal C_\mu$, represented by a measure $\eta$.
Using an embedding of $\eta$ and letting $n\to\infty$, we get
$\eta_n\to\eta$ in $W_1$ by Lemma~\ref{lem:brownian-w1}.
Continuity and linear growth of $g_i$ give
$$
  \int_\R g_i\,d\eta_n\longrightarrow\int_\R g_i\,d\eta=a_i,
  \qquad i=1,\ldots,m.
$$
This follows from the standard characterization of $W_1$ convergence by
weak convergence and uniform integrability of first moments; equivalently,
one may truncate $g_i$ and control the tails by the first moments; see
\cite{BogachevWeakConvergence}.
Applying Lemma~\ref{lem:exp-smoothing} and
Theorem~\ref{thm:markovization-intensity} to the randomized stopping time
with kernel $\delta_{\tau(\omega)}$ shows that each approximating vector
belongs to $D_0$.
Thus $D_0$ is dense in $\mathcal C_\mu$.

By Lemma~\ref{lem:ri-conv}, there are $a^1,\ldots,a^N\in D_0$ and
$\theta_j\ge0$, with $\sum_j\theta_j=1$, such that
$c=\sum_j\theta_j a^j$.
For each $j$, choose an admissible Markovian intensity $\lambda^j$ that
realizes $a^j$ and satisfies $0\le\lambda^j\le n_j$.
Let $\xi^j=\xi_{\lambda^j}$ and $\xi=\sum_j\theta_j\xi^j$.
This mixture is in $\RST$, since its first marginal is $P_\mu$ and its
survival function is adapted. Put
$$
  S_t^\xi=\sum_j\theta_j S_t^{\lambda^j},\qquad
  q_t^\xi=\sum_j\theta_j S_t^{\lambda^j}\lambda^j(t,B_t).
$$
Let $M=\max_j n_j$. Then $S_t^\xi\ge e^{-Mt}>0$ and
$0\le q_t^\xi\le M S_t^\xi$.
The process $\alpha_t=q_t^\xi/S_t^\xi$ is progressive and bounded by
$M$. Moreover,
$$
  S_t^\xi=1-\int_0^t q_s^\xi\,ds,\qquad
  \int_0^\infty q_t^\xi\,dt=1
  \quad P_\mu\text{-a.s.}
$$
The ordinary differential identity $dS_t^\xi=-S_t^\xi\alpha_t\,dt$
therefore shows that $\alpha$ is admissible and $\xi=\xi_\alpha$.

Its terminal law is a finite mixture of the terminal laws of $\xi^j$.
It has finite first moment, is above $\mu$ in convex order, and realizes
$c$.
Apply Theorem~\ref{thm:markovization-intensity} to $\alpha$.
The resulting $\lambda$ preserves the terminal law.
The estimate $\pi_\alpha\le M\overline L^\alpha$ allows its Borel
version to be chosen in $[0,M]$ by the argument used in
Theorem~\ref{thm:bounded-markov-approximation}.
This proves all assertions.
\end{proof}

\begin{example}
Let $K_1<\cdots<K_m$ and $g_i(x)=(x-K_i)^+$.
These functions are continuous and have linear growth.
Proposition~\ref{prop:finite-terminal-constraints} applies to
$$
  \mathcal C_\mu(g_1,\ldots,g_m)
  =\left\{
    \left(\int_\R(x-K_i)^+\,d\eta\right)_{i=1}^m:
    \eta\in\mathcal P_1(\R),\ \mu\cx\eta
   \right\}.
$$
Every vector in its relative interior is realized by a bounded
Markovian intensity.
\end{example}

\begin{remark}[Boundary points]
\label{rem:boundary}
The relative-interior assumption is essential.
Every intensity-generated
Brownian stopping time has a terminal law absolutely continuous with
respect to Lebesgue measure. Indeed, if $A\subset\R$ is a Lebesgue-null
Borel set, then $P_\mu(B_t\in A)=0$ for every $t>0$, because Brownian
transition laws have densities. Tonelli's theorem gives
$$
  p_{\xi_\alpha}(A)
  =\E_{P_\mu}\int_0^\infty q_t^\alpha\1_{\{B_t\in A\}}\,dt=0.
$$
Now take $\mu=\delta_0$ and $g(x)=|x|$.
Then $\mathcal C_{\delta_0}(g)=[0,\infty)$.
The boundary value $0$ is attained by the ordinary stopping time
$\tau=0$. The identity $\int|x|\,p(dx)=0$ forces $p=\delta_0$.
The absolute-continuity property above therefore excludes an
intensity-generated realization of this boundary value.
\end{remark}

\begin{remark}[The class of test functions]
\label{rem:test-functions}
Continuity is essential in the general proposition.
For $\mu=\delta_0$ and $g=\1_{\{0\}}$, define the moment region by the
same formula. It equals $[0,1]$, since the measures
$$
  \eta_a=a\delta_0+\frac{1-a}{2}(\delta_{-1}+\delta_1),
  \qquad 0\le a\le1,
$$
are above $\delta_0$ in convex order and satisfy $\int g\,d\eta_a=a$.
Thus $1/2$ lies in the relative interior of the region.
Every intensity-generated terminal law assigns zero mass to $\{0\}$ by
Remark~\ref{rem:boundary}. Thus $1/2$ lies outside the set of values
generated by intensities. For arbitrary discontinuous functions, $W_1$
convergence leaves the corresponding integrals uncontrolled.
\end{remark}

\bibliographystyle{plain}
\bibliography{references}
\end{document}